\documentclass{amsart}

\input xypic 
\input xy 
\xyoption{all} 
\usepackage{amssymb} 
\usepackage{bbm}
\usepackage{tikz-cd}

\usepackage{float}
\usepackage{graphicx}
\usepackage{accents}
\usepackage[bookmarksnumbered, bookmarksopen, 
colorlinks,citecolor=blue,linkcolor=blue,backref]{hyperref}

\newtheorem{theorem}{Theorem}[section]
\newtheorem{proposition}[theorem]{Proposition}
\newtheorem{lemma}[theorem]{Lemma}

\theoremstyle{definition}

\newtheorem{remark}[theorem]{Remark}
\newtheorem{example}[theorem]{Example}

\newcounter{RomanNumber}

\newcommand{\testcap}{\mathrel{\vcenter{\offinterlineskip
\hbox{$\cap$}\vskip-1.8ex\hbox{$\kern0.2em | \kern0.15em$}}}}

\newcounter{bean}

\newcommand{\larrow}{\relbar\!\!\relbar\!\!\rightarrow}

\newcommand{\qqed}{\hfill\Box}

\newcommand{\be}{\begin{equation}}
\newcommand{\ee}{\end{equation}}

\begin{document}

%%% Title

\title{An Almost Flat Spin$^c$ Manifold Bounds}

\author{Fei Han}
\address{Department of Mathematics,
National University of Singapore, Singapore 119076}
%\curraddr{}
\email{mathanf@nus.edu.sg}
%\thanks{}
%\urladdr{http://www.math.nus.edu.sg/~mathanf}

\author{Ruizhi Huang} 
\address{State Key Laboratory of Mathematical Sciences \& Institute of Mathematics, Academy of Mathematics and Systems Science, 
   Chinese Academy of Sciences, Beijing 100190, China} 
\email{huangrz@amss.ac.cn} 
   \urladdr{https://sites.google.com/site/hrzsea/}

\author{Weiping Zhang}
\address{Weiping Zhang, Chern Institute of Mathematics \& LPMC, Nankai
University, Tianjin 300071, China.}
\email{weiping@nankai.edu.cn}


%%% Abstract

\begin{abstract} 
We prove that every almost flat spin\(^c\) manifold bounds a compact orientable manifold, 
thereby settling, in the spin\(^c\) case, a long-standing conjecture of Farrell--Zdravkovska and S.~T.~Yau.
\end{abstract}

\maketitle

\tableofcontents

%%%%%%%%%%%%%%%%%%%%%%%%%%%%%%%%%%%%%%%%%%%%%%%%%%%%%%%%%%%
\section{Introduction} 

A closed manifold \( M \) is said to be \emph{almost flat} if, for every \( \epsilon > 0 \), there exists a Riemannian metric \( g_\epsilon \) on \( M \) such that 
\[
\operatorname{diam}(M, g_\epsilon) \le 1, \qquad |K_{g_\epsilon}| < \epsilon ,
\]
where \( K_{g_\epsilon} \) denotes the sectional curvature of \( g_\epsilon \). 
Thus, \( M \) admits metrics of arbitrarily small curvature and uniformly bounded diameter. For each dimension \( n \), there exists a constant \( \epsilon_n > 0 \) with the property that if an \( n \)-dimensional manifold admits an \( \epsilon_n \)-flat metric of diameter at most one, then it is almost flat \cite{CG86}.

A classical theorem of Gromov~\cite{Gro78} asserts that every almost flat manifold is finitely covered by a nilmanifold. Conversely, Farrell and Hsiang~\cite{FH83} proved that any manifold finitely covered by a nilmanifold is homeomorphic to an almost flat manifold. Ruh~\cite{Ruh82} subsequently strengthened Gromov’s result, showing that every almost flat manifold is in fact diffeomorphic to an infranilmanifold.

A long-standing conjecture, proposed independently by Farrell--Zdravkovska~\cite{FZ83} and Yau~\cite{Yau93}, predicts that every almost flat manifold bounds a compact manifold. 
In the special case of flat manifolds, this was established by Hamrick and Royster~\cite{HR82} in a striking theorem. 
More recently, Davis and Fang~\cite{DF16} confirmed the conjecture under the additional assumption that the \(2\)-Sylow subgroup of the holonomy group is cyclic or generalized quaternionic; their result was later rederived by Xiong~\cite{Xio16} through a different argument. 
In particular, the theorem of Davis and Fang encompasses earlier advances due to Farrell--Zdravkovska~\cite{FZ83} and Upadhyay~\cite{Upa01}.

The general case of the conjecture remains open. Moreover, Davis and Fang~\cite{DF16} pointed out that determining whether every almost flat \emph{spin} manifold, possibly after changing the spin structure, bounds a spin manifold is a subtle and difficult problem. It is therefore natural, and of considerable interest, to ask the corresponding question for almost flat \(\mathrm{spin}^c\) manifolds.

The main result of this paper provides a confirmation of the Farrell--Zdravkovska--Yau conjecture when the almost flat manifold is spin\(^c\) as follows. 

\begin{theorem}\label{spinc-thm-intro}
Let $M$ be an \(n\)-dimensional almost flat spin\(^c\) manifold. Then $M$ bounds an orientable manifold. Moreover, if \(n\) is odd, then $M$ with any spin$^c$ structure bounds a spin\(^c\) manifold.
\end{theorem}

The existence of a spin\(^c\) structure on an almost flat manifold is intimately connected with the nature of its holonomy group (which is finite). 
We now consider a particular instance of this relationship and obtain the following theorem as a consequence of Theorem~\ref{spinc-thm-intro}. 
Recall that the classical \emph{Schur multiplier} of a group \(G\) is its second group homology \(H_2(G;\mathbb{Z})\), introduced by Issai Schur in 1904. It is clear that the Schur multiplier of a finite group is finite.

\begin{theorem}\label{G-thm-intro}
Let \(M\) be an almost flat oriented manifold with holonomy group \(G\).
If the Schur multiplier of \(G\) is of odd order, then \(M\) bounds a compact manifold.  

In particular, if the \(2\)-Sylow subgroup of \(G\) is cyclic, generalized quaternionic, or coincides with the \(2\)-Sylow subgroup of any finite simple group listed in Example~\ref{simple-ex}, then \(M\) is the boundary of a compact manifold.
\end{theorem}

Theorem~\ref{G-thm-intro} combined with Theorem \ref{spinc-thm-intro} extends the result of Davis and Fang~\cite{DF16} for almost flat {\em oriented} manifolds. Additionally, as noted in Example~\ref{Xio-ex}, the result of Xiong~\cite{Xio16} appears as a special case of Theorem~\ref{spinc-thm-intro}.

\medskip

In dimension 4, we can give a complete verification of the Farrell--Zdravkovska--Yau conjecture, both in the orientable and the spin cases.
\begin{theorem}\label{4-thm-intro}
Let \(M\) be an almost flat orientable \(4\)-manifold. Then \(M\) bounds a spin\(^c\) manifold. Furthermore, if \(M\) is spin, then \(M\) with any spin structure bounds a spin manifold.
\end{theorem}

Ontaneda~\cite{Ont20} established an important converse to the classical Cheeger--Fukaya--Gromov theorem~\cite{CFG92}: 
if an almost flat manifold \(M\) bounds, then there exists a compact manifold \(W\) with boundary \(M\) such that \(W \setminus M\) carries a complete, finite-volume Riemannian metric with negatively pinched sectional curvature.  

Farrell and Zdravkovska~\cite{FZ83} went further to conjecture that an almost flat (respectively, flat) manifold bounds a compact manifold whose interior admits a complete, finite-volume metric with negative (respectively, constant negative) sectional curvature. 
In the flat case this conjecture was disproved by Long and Reid~\cite{LR20}, while in the almost flat case it was verified by Davis and Fang~\cite{DF16} when the holonomy group is cyclic or quaternionic.  

In the present work, combining Theorem~\ref{spinc-thm-intro} with Ontaneda’s theorem~\cite{Ont20}, we confirm the conjecture for almost flat spin\(^c\) manifolds.

\begin{theorem}
An almost flat spin\(^c\) manifold bounds a compact manifold whose interior admits a complete, finite-volume Riemannian metric with negatively pinched sectional curvature. 
\end{theorem}

$\, $

Our approach to the proofs of the foregoing theorems departs in a substantial way from earlier work. Whereas infranilmanifolds have traditionally been studied largely through group-theoretic methods, we adopt a more topological point of view. The central tool in our proof is the integral Wu class for 
spin\(^c\)
 manifolds, developed in Section~\ref{sec: spincWu}. This construction is inspired by the integral spin Wu class introduced by Hopkins and Singer~\cite{HS05}. The guiding idea is that, while the Stiefel–Whitney classes themselves are only indirectly related to geometry, the integral lift of Wu classes admits a more transparent geometric interpretation. As a result, questions concerning the vanishing of Stiefel–Whitney numbers may be reduced to parity properties of these integral Wu classes.

The essential point in the proof of the main theorem for 
spin\(^c\) almost flat manifolds is that the Atiyah–Singer index theorem yields strong divisibility properties for the top power of the 
spin\(^c\) characteristic class $c$. This divisibility, in turn, guarantees the evenness of the integral Wu numbers, thereby forcing the vanishing of the relevant Stiefel–Whitney numbers.

$\, $

The paper is organized as follows.  
In Section~\ref{sec: spincWu}, we recall the integral spin Wu class of Hopkins and Singer and introduce the corresponding integral Wu class for spin\(^c\) vector bundles, establishing along the way an estimate (Lemma~\ref{denom-lemma}) that will be used later.  
Section~\ref{sec: spincbound} contains the proof of our main theorem, Theorem~\ref{spinc-thm-intro}, which confirms the Farrell--Zdravkovska--Yau conjecture for almost flat spin\(^c\) manifolds.  
In Section~\ref{sec: holo}, we derive Theorem~\ref{G-thm-intro} as an application of Theorem~\ref{spinc-thm-intro}.  
Finally, Section~\ref{sec: 4} is devoted to the 4-dimensional case, where Theorem~\ref{4-thm-intro} is proved.

\bigskip 

\noindent{\bf Acknowledgements.} 
Fei Han was partially supported by the grant AcRF A-8000451-00-00 from National University of Singapore. Ruizhi Huang was supported in part by the National Natural Science Foundation of China (Grant nos. 12331003 and 12288201), the National Key R\&D Program of China (No. 2021YFA1002300) and the Youth Innovation Promotion Association of Chinese Academy Sciences. Weiping Zhang was partially supported by NSFC Grant No. 11931007, National Key R\&D Program of China (2024YFA1013202) and Nankai Zhide Foundation.

The authors would like to thank Haibao Duan for helpful discussions on integral Wu classes. Ruizhi Huang also thanks Xiaolei Wu for valuable conversations on group homology, and Nansen Petrosyan for pointing out Remark \ref{nasen}.

%----------------------------------------------------------------------------------------------------------------------------------------------------------------------------------------------------------%
\section{Integral Wu class for spin$^c$ manifolds}
\label{sec: spincWu}

The purpose of this section is to construct the integral Wu class for spin\(^c\) manifolds, which will serve as our principal tool in the study of almost flat spin\(^c\) manifolds.  

We begin by recalling the integral Wu class for spin manifolds as introduced by Hopkins and Singer.  
Let \(V\) be an \(n\)-dimensional oriented vector bundle over a compact, oriented, smooth manifold \(M\).  
The bundle \(V\) is said to be {\it spin} if and only if its second Stiefel--Whitney class \(\omega_2(V)\) vanishes.

In \cite[Appendix E.1]{HS05}, Hopkins-Singer constructed an integral lift of the mod-$2$ {\it normal} Wu class (i.e. the Wu class of the stable complement) for a spin vector bundle $V$, denoted by 
\[
\nu _{t}^{Spin}(V)=1+\nu _{4}^{Spin}(V)+\nu _{8}^{Spin}(V)+\cdots \in H^{4\ast}(M; \mathbb{Z}). 
\]
In the universal case, the {\it integral spin normal Wu class} 
\[
\nu _{t}^{Spin}=1+\nu _{4}^{Spin}+\nu _{8}^{Spin}+\cdots  \in H^{4\ast}(BSpin; \mathbb{Z}) 
\]
is determined by the characteristic series
\[
g(x)=\sqrt{f(x)f(-x)}=1-\frac{x^{2}}{2}-\frac{9x^{4}}{8}-\frac{%
17x^{6}}{16}-\frac{277x^{8}}{128}-\frac{839x^{10}}{256}+\cdots \in \mathbb{Z}%
\left[ \frac{1}{2}\right] [[x]],
\]
where 
\[
f(x)=1+x+x^{3}+\cdots +x^{2^{n}-1}+\cdots 
\]
is the characteristic series of the mod-$2$ normal Wu class. In particular, the mod-$2$ reduction of $\nu _{t}^{Spin}$ is the mod-$2$ normal Wu class. 

In this paper, we work with tangential Wu classes, that is, the usual Wu classes associated with the tangent bundle. We use the term “tangential” solely to distinguish them from normal Wu classes. Following the argument in \cite[Appendix E.1]{HS05}, one can construct a family of integral lifts of the mod-$2$ tangential Wu class for a spin vector bundle $V$; see \cite[Section 2]{HHZ25} for details. Any such lift suffices for the applications in this paper, and we fix one choice throughout.

There is an integral lift of the mod-$2$ {\it tangential} Wu class for a spin vector bundle $V$, denoted by 
\[
\mu _{t}^{Spin}(V)=1+\mu _{4}^{Spin}(V)+\mu _{8}^{Spin}(V)+\cdots \in H^{4\ast}(M; \mathbb{Z}), 
\]
where in the universal case, the {\it integral spin tangential Wu class} 
\[
\mu _{t}^{Spin}=1+\mu _{4}^{Spin}+\mu _{8}^{Spin}+\cdots  \in H^{4\ast}(BSpin; \mathbb{Z}) 
\]
is determined by the characteristic series
\[
G(x)=\sqrt{h(x)h(-x)}=1+\frac{x^{2}}{2}+\frac{11x^{4}}{8}+\frac{%
5x^{6}}{16}+\frac{51x^{8}}{128}+\frac{95x^{10}}{256}+\cdots \in \mathbb{Z}%
\left[ \frac{1}{2}\right] [[x]],
\]
where 
\[
h(x)=1+x+x^{2}+\cdots +x^{2^{n}}+\cdots
\]
is the characteristic series of the mod-$2$ tangential Wu class \cite{Wu50, HBJ94}. In particular, the mod-$2$ reduction of $\mu _{t}^{Spin}$ is the mod-$2$ tangential Wu class.

See \cite[Appendix E.1]{HS05} and \cite[Example 2.3]{HHZ25} for the explicit expressions of the first few integral spin normal and tangential Wu classes in terms of Pontryagin classes, respectively. 

Since we restrict attention to tangential Wu classes throughout the remainder of the paper, we shall simply refer to them as Wu classes.

\medskip

Now let us turn to the spin$^c$ cases. Let $V$ be an $n$-dimensional oriented vector bundle over a compact oriented smooth manifold $M$. It is said to be {\it spin$^c$}  if and only if its second Stiefel-Whitney class $\omega_2(V)$ lies in the image of the mod-$2$ reduction homomorphism
\[
\rho_2: H^2(M;\mathbb{Z}) \rightarrow H^2(M; \mathbb{Z}/2).
\]
Specifying a spin$^c$ structure is equivalent to choosing a particular class $c\in H^2(M;\mathbb{Z})$ such that $\rho_2(c)=\omega_2(V)$, which determines, and is determined by a complex line bundle $\xi$ with its associated circle bundle
\[
S^1\rightarrow S(\xi)\rightarrow M
\]
and a spin structure on $V\oplus \xi$. 
We shall often refer to a spin$^c$ vector bundle $V$ together with its characteristic class $c=c_1(\xi)$ simply as the pair $(V,\xi)$ or $(V,c)$, and write $(M,c)$ when $V=TM$ is the tangent bundle of $M$.

We define an integral spin$^c$ (tangential) Wu class through the integral spin Wu class. To this purpose, we need to fix a choice of complex Wu classes. In \cite[Appendix E.1.2]{HS05}, Hopkins-Singer defined a {\it complex normal Wu class} for complex vector bundles by the characteristic class $f(x)$, whose mod-$2$ reduction is the mod-$2$ normal Wu class. Similarly, we can define a {\it complex (tangential) Wu class} 
\[
\mu _{t}^{U}(\eta)= 1+\mu _{2}^{U}(\eta)+\mu _{4}^{U}(\eta)+\cdots \in H^{2\ast}(M; \mathbb{Z})
\]
for a complex vector bundle $\eta$ over $M$ by the characteristic series $h(x)$, whose mod-$2$ reduction is the mod-$2$ (tangential) Wu class. 
 
For a spin$^c$ bundle $(V, \xi)$, define the {\it integral spin$^c$ (tangential) Wu class} by 
\[
\mu _{t}^{Spin^c}(V):=\frac{\mu _{t}^{Spin}(V\oplus \xi)}{\mu _{t}^{U}(\xi)}=\frac{\mu _{t}^{Spin}(V\oplus \xi)}{h_t(c_1(\xi))}.
\]
Since the mod-$2$ reduction of $\mu _{t}^{Spin}(V\oplus \xi)$ and $\mu _{t}^{U}(\xi)$ are the mod-$2$ Wu classes of $V\oplus \xi$ and $\xi$, respectively, the Whitney product formula implies that the mod-$2$ reduction of $\mu _{t}^{Spin^c}(V)$ is the mod-$2$ Wu classes of $V$. This justifies the definition.

The Chern root algorithm can be applied to express the universal integral spin$^c$ Wu classes $\mu_{2k}^{Spin^c}$ as polynomials in the Pontryagin classes together with the universal characteristic class $c \in H^2(BSpin^c;\mathbb{Z})$.

\begin{example}
The first eight integral spin$^c$ Wu classes rationally are
\[
\begin{split}
\mu_{2}^{Spin^c}&=-c,\\
\mu_{4}^{Spin^c}&=\frac{c^2}{2}+\frac{p_1}{2},\\
\mu_{6}^{Spin^c}&=\frac{c^3}{2}-\frac{cp_1}{2},\\
\mu_{8}^{Spin^c}&=-\frac{5c^4}{8} +\frac{c^2p_1}{4} + \frac{11p_1^2}{8}-\frac{5p_2}{2},\\
\mu_{10}^{Spin^c}&=\frac{9c^5}{8} + \frac{c^3p_1}{4} - \frac{11cp_1^2}{8}+\frac{5cp_2}{2},\\
\mu_{12}^{Spin^c}&=-\frac{11c^6}{16} - \frac{5c^4p_1}{16}  + \frac{11c^2p_1^2}{16} +\frac{5p_1^3}{16}-\frac{5c^2p_2}{4}-\frac{p_1p_2}{4}-p_3,\\
\mu_{14}^{Spin^c}&=-\frac{15c^7}{16} + \frac{9c^5p_1}{16}  + \frac{11c^3p_1^2}{16} -\frac{5cp_1^3}{16}-\frac{5c^3p_2}{4}+\frac{cp_1p_2}{4}+cp_3,\\
\mu_{16}^{Spin^c}&=\frac{211c^8}{128} - \frac{11c^6p_1}{32}  - \frac{55c^4p_1^2}{64} +\frac{5c^2p_1^3}{32}+\frac{51p_1^4}{128}+\frac{25c^4p_2}{16}-\frac{c^2p_1p_2}{8}  \\
&  \ \ \ \  -\frac{23p_1^2p_2}{16}+\frac{19p_2^2}{8}-
\frac{c^2p_3}{2}-2p_1p_3+\frac{3p_4}{2}.
\end{split}
\]
\end{example}

The estimation in the following lemma is optimal by comparing to the previous example.
\begin{lemma}\label{denom-lemma}
In the expression of the universal integral spin$^c$ Wu class $\mu_{2n}^{Spin^c}$, the denominator of the coefficient $a_n$ of $c^{n}$ satisfies 
\[
{\rm denom}~(a_n)~|~2^{n-s_2(n)},
\] 
where $s_2(n)$ denotes the sum of the binary digits of $n$.
\end{lemma}
\begin{proof}
To compute the coefficient of $c^n$, we may formally suppose that there is a universal bundle $V$ such that all the Pontryagin classes $p_i(V)$ vanish. Then  we have 
\[
\mu _{t}^{Spin^c}(V) 
=\frac{\mu _{t}^{Spin}(V\oplus \xi)}{\mu _{t}^{U}(\xi)}
=\frac{\sqrt{h(c)h(-c)}}{h(c)}=\sqrt{\frac{h(-c)}{h(c)}},
\]
where $h(c)=1+\sum\limits_{k=0}^{\infty}c^{2^k}$. Write 
\[
\mu _{t}^{Spin^c}(V)= \sum\limits_{k=0}^{\infty} a_k c^k\ \ \  {\rm and} \ \ \ \ 
\frac{h(-c)}{h(c)}=\sum\limits_{k=0}^{\infty} b_k c^k. \ \ \ 
\]
Note that $a_0=b_0=1$, and each $b_k\in 2\mathbb{Z}$ as $\sum\limits_{k=0}^{\infty} b_k c^k=\frac{h(-c)}{h(c)}\equiv 1~{\rm mod}~2$. 
We want to show that $2^{n-s_2(n)} a_n\in \mathbb{Z}$ for each $n$.

The previous equalities imply that $(\sum\limits_{k=0}^{\infty} a_k c^k)^2=\sum\limits_{k=0}^{\infty} b_k c^k$, and then $\sum\limits_{k=0}^{n} a_k a_{n-k}= b_n$. Accordingly, we have the recursive formula
\[
2a_n=b_n-\sum\limits_{k=1}^{n-1} a_k a_{n-k}.
\]
By induction we may suppose that $2^{k-s_2(k)} a_k\in \mathbb{Z}$ for any $k<n$. Since\footnote{
	The classical Legendre’s formula states that for a prime $p$, the exponent of $p$ in the
	prime factorization of $m!$ is
	\[
	\nu_p(m!)=\sum_{k\ge1}\left\lfloor \frac{m}{p^k}\right\rfloor
	=\frac{m-s_p(m)}{p-1},
	\]
	where $s_p(m)$ denotes the sum of the digits of $m$ in base $p$ (for instance see \cite[page 77]{Mo12}). Specialize to $p=2$, one has $\nu_2(n!)=n-s_2(n)$. Since $\frac{n!}{k!(n-k)!}$ is an integer, one has $\nu_2(n!)\geq \nu_2(k!)+\nu_2((n-k)!)$, which implies that $s_2(n)\leq s_2(k)+s_2(n-k)$.
} 
$$s_2(n)\leq s_2(k)+s_2(n-k), $$  it follows that  
\[
  2^{n-s_2(n)} a_k a_{n-k}=  2^{s_2(k)+s_2(n-k)-s_2(n)} \cdot (2^{k-s_2(k)} a_k)   \cdot (2^{n-k-s_2(n-k)} a_{n-k}) 
  \in \mathbb{Z},
\]
and then 
\[
2^{n+1-s_2(n)} a_n=2^{n-s_2(n)} b_n -\sum\limits_{k=1}^{n-1}\Big( 2^{n-s_2(n)} a_k a_{n-k} \Big) \in \mathbb{Z}.
\]
To prove that $2^{n-s_2(n)} a_n\in \mathbb{Z}$, it is equivalent to show that the right hand of the preceding equality is an even integer, and is further equivalent to the sum $2^{n-s_2(n)}\sum\limits_{k=1}^{n-1} a_k a_{n-k} \in 2\mathbb{Z}$.

Observe that in the sum $\sum\limits_{k=1}^{n-1} a_k a_{n-k}$, each term $a_k a_{n-k}$ appears twice except the term $a_{m}^2$ when $n=2m$. In the latter case, 
\[
2^{n-s_2(n)}a_{m}^2=2^{2m-s_2(m)}a_{m}^2 =   2^{s_2(m)}   \cdot (2^{m-s_2(m)} a_m)^2 \in 2\mathbb{Z}.
\]
Hence, in any case $2^{n-s_2(n)}\sum\limits_{k=1}^{n-1} a_k a_{n-k} \in 2\mathbb{Z}$. It follows that $2^{n+1-s_2(n)} a_n\in 2\mathbb{Z}$ and then $2^{n-s_2(n)} a_n\in \mathbb{Z}$.
\end{proof}

%----------------------------------------------------------------------------------------------------------------------------------------------------------------------------------------------------------%
\section{Spin$^c$ almost flat manifolds bound}
\label{sec: spincbound}

Let $(M, c)$ be an $n$-dimensional closed spin$^c$ smooth manifold. Then $M$ has rational characteristic numbers of the form 
\[
\langle \lambda(c, p_1, \cdots, p_{\lfloor \frac{n}{4} \rfloor }),  \ [M]\rangle,
\]
where $\lambda$ is a rational polynomial in $c$ and the Pontryagin classes $p_i(M)$ of $M$, homogeneous of total degree $n$. We refer to any such characteristic number as a {\it $c$-Pontryagin number} of $(M,c)$.

The following lemma characterizes oriented and spin$^c$ cobordisms in terms of characteristic classes; see \cite{Wal60} for the oriented case and \cite[Theorem, p.~337]{Sto68} for the spin$^c$ case. 
\begin{lemma}\label{spinc-bord-lemma}
Two closed oriented manifolds are oriented cobordant if and only if they have the same Stiefel-Whitney numbers and Pontryagin numbers. 

Two closed spin$^c$ manifolds are spin$^c$ cobordant if and only if they have the same Stiefel-Whitney numbers and $c$-Pontryagin numbers. ~$\qqed$ 
\end{lemma}

\begin{lemma}\label{chern-lemma}
Let $M$ be an almost flat manifold. Then the rational Pontryagin classes of $M$ all vanish. 
\end{lemma}
\begin{proof}
For the almost flat manifold $M$, by a classical result of Gromov \cite{Gro78} there is a finite $d$-cover $\widetilde{M}\stackrel{\pi}{\larrow} M$, where $\widetilde{M}$ is a nilmanifold. Then 
$
{\rm tr}\circ \pi^\ast(p_i(M))=d\cdot p_i(M)
$, where $H^\ast(\widetilde{M};\mathbb{Z})\stackrel{{\rm tr}}{\larrow} H^\ast(M;\mathbb{Z})$ is the transfer map of a finite cover. Since a nilmanifold is parallelizable, all Pontryagin classes of $\widetilde{M}$ vanish integrally. It follows that all rational Pontryagin classes of $M$ vanish. 
\end{proof}

\begin{lemma}\label{index-lemma}
Let $(M, c)$ be a $2n$-dimensional almost flat spin$^c$ manifold. Then $2^n n!~|~\langle c^{n}, [M]\rangle$.
\end{lemma}
\begin{proof} Consider the spin$^c$ Dirac operator $D^c$ on $M$. The Atiyah-Singer index theorem asserts that
\[
 \mathrm{Ind}(D^c)=\int_M\widehat{A}(M)e^{c/2} \in \mathbb{Z}.
 \]
Since all the rational Pontryagin classes of $M$ vanish, we see that 
\[
 \mathrm{Ind}(D^c)=\int_M\widehat{A}(M)e^{c/2}=\int_M\frac{c^n}{n!\cdot 2^n}.
  \]
 The desired divisibility follows. 
\end{proof}

Following \cite[Chapter 16]{MS75}, it is convenient to denote $S_{I}=S_{j_1}\cdots S_{j_r}$ for any multiplicative symbol system $\{S_{j}\}_{j\in \mathbb{Z}^{+}}$ indexed by positive integers and any partition $I=\{j_1,\ldots, j_r\}$ of a fixed positive integer $m$.

\begin{lemma}\label{sw-lemma}
Let $(M, c)$ be an $n$-dimensional almost flat spin$^c$ manifold. Then its Stiefel-Whitney numbers all vanish.
\end{lemma}
\begin{proof}
Since the rational Pontryagin classes of the almost flat manifold $M$ all vanish by Lemma \ref{chern-lemma}, each integral spin$^c$ Wu classes $\mu _{2i}^{Spin^c}(M)=a_i c^i$, where $a_i$ is the coefficient of $c^i$ in the expression of the universal integral spin$^c$ Wu class $\mu _{2i}^{Spin^c}$. Hence, any integral spin$^c$ Wu number vanishes if $n$ is odd, or is of the form $a_I c^{m}$ with $I=\{j_1,\ldots, j_r\}$ a partition of $m$ if $n=2m$ is even. In the latter case, recall the classical Legendre formula 
$$m!=2^{m-s_2(m)}\cdot (\text{an odd number})$$
 and the fact $s_2(m)\leq s_2(j_1)+\cdots+s_2(j_r)$. Then by Lemmas \ref{denom-lemma} and \ref{index-lemma}, 
\[
\langle a_I c^{m}, [M]\rangle \in 2^m m! a_I \mathbb{Z}\subseteq  2^{s_2(j_1)+\cdots+s_2(j_r)}m!  \mathbb{Z}  \subseteq 
2^{m}\mathbb{Z}.
\]
In particular, the mod-$2$ reduction of any integral spin$^c$ Wu number of $M$ vanishes, or equivalently, the mod-$2$ Wu numbers of $M$ all vanish. Since Stiefel-Whitney numbers are equivalent to mod-$2$ Wu numbers, the Stiefel-Whitney numbers of $M$ all vanish as well. 
\end{proof}

\begin{remark}\label{c-remark}
In the proof of Lemma \ref{sw-lemma}, we have showed that, for a $2m$-dimensional almost flat spin$^c$ manifold $M$,
\[
2^m~|~\langle \mu _{2I}^{Spin^c}(M), \   [M] \rangle
\]
for any partition $I$ of $m$.
\end{remark}

The following proposition provides a simple criterion when almost flat spin$^c$ manifold bounds.  
\begin{proposition}\label{cn=0-prop}
Let $(M, c)$ be an $n$-dimensional almost flat spin$^c$ manifold. Then $(M, c)$ bounds a spin$^c$ manifold if and only if $\langle c^{\lfloor \frac{n}{2} \rfloor }, [M]\rangle =0$.
\end{proposition}
\begin{proof}
By Lemma \ref{spinc-bord-lemma}, $(M, c)$ bounds a spin$^c$ manifold if and only if the Stiefel-Whitney numbers and $c$-Pontryagin numbers all vanish. By Lemma \ref{sw-lemma}, its Stiefel-Whitney numbers all vanish. By Lemma \ref{chern-lemma}, only the number $\langle c^{\lfloor \frac{n}{2} \rfloor }, [M]\rangle$ can be nontrivial among all the $c$-Pontryagin numbers. Hence, the proposition follows. 
\end{proof}

We are in the place to prove Theorem \ref{spinc-thm-intro}.
\begin{proof}[Proof of Theorem \ref{spinc-thm-intro}]
If $n$ is even, the Stiefel-Whitney numbers of $M$ all vanish by Lemma \ref{sw-lemma}, while the Pontryagin numbers of $M$ all vanish by Lemma \ref{chern-lemma}. By Lemma \ref{spinc-bord-lemma} it follows that $(M, c)$ bounds an orientable manifold. 

If $n$ is odd, then $\langle c^{\lfloor \frac{n}{2} \rfloor }, [M]\rangle =0$ automatically. By Proposition \ref{cn=0-prop} it follows that $(M, c)$ bounds a spin$^c$ manifold.
\end{proof}

\begin{remark}\label{nasen}
	The proof in fact yields a more general result than Theorem \ref{spinc-thm-intro}. Indeed, we show that if a spin$^c$ manifold $M$ has all rational Pontryagin classes vanishing, then it bounds an orientable manifold. Moreover, if $n$ is odd, then $M$, equipped with any spin$^c$ structure, bounds a spin$^c$ manifold.
	\end{remark}

%----------------------------------------------------------------------------------------------------------------------------------------------------------------------------------------------------------%
\section{Special Holonomies imply spin$^c$}
\label{sec: holo}

Let $M$ be an $n$-dimensional almost flat oriented manifold. By the results of Gromov \cite{Gro78} and Ruh \cite{Ruh82}, $M$ is diffeomorphic to an infranilmanifold $\Gamma \backslash L\rtimes G\slash G$, where $L$ is a simply connected nilpotent Lie group, $G$ is a finite subgroup of ${\rm Aut}(L)$, and $\Gamma$ is a discrete torsion-free cocompact subgroup of $L\rtimes G$ with a short exact sequence 
\[
1\larrow \Gamma \cap L \larrow \Gamma \larrow G\larrow 1.
\]
 The finite group $G$ is called the {\it holonomy group} of the almost flat manifold $M$. 

Let $\varphi: M\larrow BSO(n)$ be the classifying map of the tangent bundle $TM$. It is well known that $\varphi$ factors through the classifying space $BG$ of $G$:
\[
\varphi: M\stackrel{\eta}{\larrow} BG\stackrel{}{\larrow} BSO(n),
\]
where $\eta$ is the classifying map of the regular covering $\Gamma \cap L\backslash L \larrow \Gamma \backslash L\rtimes G\slash G$; see, for instance, \cite[Section 3]{Xio16}.

The following lemma establishes a sufficient condition under which $M$ is spin$^c$. 
For a finite group $A$, denote by ${\rm Syl}_2(A)$ the $2$-Sylow subgroup of $A$. Recall $H^\ast(A)=H^\ast(BA)$ and $H_\ast(A)=H_\ast(BA)$ are the group cohomology and homology of $A$, respectively.  
\begin{lemma}\label{H3lemma}
If $H^3({\rm Syl}_2(G); \mathbb{Z})=0$, then the almost flat manifold $M$ is spin$^c$.
\end{lemma}
\begin{proof}
Consider the following homotopy lifting problem 
\[
\diagram
&& BSpin^c(n) \dto^{}  \\
M\rto^{\eta} & BG \rto^{} \ar@{.>}[ur]  & BSO(n) \dto^{W_3}\\
&& K(\mathbb{Z}, 3),
\enddiagram
\]
where the rightmost column is the standard homotopy fibration for spin$^c$ structure with the map $W_3$ representing the third integral Stiefel-Whitney class $W_3\in H^3(BSO(n); Z)$. 

Suppose that $H^3(G; \mathbb{Z})=0$. Then the composite $BG\larrow BSO(n)\stackrel{W_3}{\larrow} K(\mathbb{Z}, 3)$ is null homotopic, and the map $BG\larrow BSO(n)$ can be lifted to the homotopy fibre $BSpin^c(n)$. By the preceding diagram, it follows that $\varphi$ can be lifted to $BSpin^c(n)$. Hence, $M$ is spin$^c$ if $H^3(G; \mathbb{Z})=0$.

Suppose that $H^3({\rm Syl}_2(G); \mathbb{Z})=0$. There is an odd-degree finite covering $\widehat{M}\longrightarrow M$ of almost flat manifolds such that the holonomy group of $\widehat{M}$ is ${\rm Syl}_2(G)$. The previous argument shows that $\widehat{M}$ is spin$^c$, or equivalently $W_3(\widehat{M})=0$. However, since the class $W_3$ is the Bockstein of the second Stiefel-Whitney class $\omega_2 \in H^2(BSO(n);\mathbb{Z}/2)$, the class $W_3$ is of order $2$ as $\omega_2$ is. Therefore, the odd-degree finite covering $\widehat{M}\larrow M$ implies that $W_3(M)=0$, that is, $M$ is spin$^c$.
\end{proof}

With the following lemma, Lemma \ref{H3lemma} applies to the case considered by Davis and Fang \cite{DF16}.
\begin{lemma}\label{DFlemma}
If ${\rm Syl}_2(G)$ is cyclic or generalized quaternionic, then $H^3({\rm Syl}_2(G);\mathbb{Z})=0$.
\end{lemma}
\begin{proof}
Based on the work of Hopf \cite{Hop45} and Seifert-Threlfall \cite{ST31}, Milnor \cite{Mil57} listed the collection of all finite groups acting freely and orthogonally on the $3$-sphere $S^3$. Among them, the $2$-groups are the cyclic groups $\mathbb{Z}/2^r$ with $r\geq 1$ and the generalized quaternionic groups $Q_{2^s}$ with $s\geq 2$. To prove the lemma, it suffices to show that $H^3(A;\mathbb{Z})=0$ for any finite group $A$ in Milnor's list.

Since $A$ acts freely and orthogonally on $S^3$, we have the spherical space form $N=S^3/A$. It is clear that $\pi_1(N)\cong A$ and $\pi_i(N)\cong \pi_i(S^3)$ for any $i\geq 2$. In particular, $\pi_2(N)=0$. By attaching cells of dimension greater than $3$, we can obtain an inclusion $j: N\hookrightarrow K(A, 1)=BA$, which indicates that $N$ is the $3$-skeleton of $BA$. Therefore, the induced map $j^\ast: H^3(A;\mathbb{Z})\larrow H^3(N;\mathbb{Z})\cong \mathbb{Z}$ is injective. However, $H^3(A;\mathbb{Z})$ is a finite group as $A$ is. It follows that $H^3(A;\mathbb{Z})=0$.
\end{proof}

The following proposition recovers the main theorem of \cite{DF16} for the {\em oriented case}. 
\begin{proposition}\label{DFprop}
Let $M$ be an $n$-dimensional almost flat oriented manifold. If the $2$-Sylow subgroup ${\rm Syl}_2(G)$ of the holonomy group $G$ of $M$ is cyclic or generalized quaternionic, then $M$ bounds an orientable manifold.
\end{proposition}
\begin{proof}
The assumption that ${\rm Syl}_2(G)$ is cyclic or generalized quaternionic implies $H^3({\rm Syl}_2(G);\mathbb{Z})=0$ by Lemma \ref{DFlemma}, and then implies that $M$ is spin$^c$ by Lemma \ref{H3lemma}. Hence, $M$ bounds by Theorem \ref{spinc-thm-intro}.
\end{proof}

\begin{example}\label{Xio-ex}
Theorem \ref{spinc-thm-intro} also recovers \cite[Theorem 1.2]{Xio16}, where $L$ is assumed to be a simply connected $2$-step nilpotent Lie group with ${\rm dim}[L,L] = 1$ and $G$ acts trivially on the center of $L$. 
Indeed, as shown in the proof of \cite[Theorem 1.2]{Xio16} there, the classifying map $\varphi: M\larrow BSO(n)$ factors as 
\[
\varphi: M\stackrel{\eta}{\larrow} BG\stackrel{}{\larrow} BU(m)\stackrel{}{\larrow} BSO(n)
\]
for a suitable unitary group $U(m)$. Consequently, $M$ is spin$^c$ and Theorem \ref{spinc-thm-intro} applies to show that $M$ bounds. 
\end{example}

Lemma \ref{H3lemma} can be reformulated using the following result.

\begin{lemma}\label{H3H2lemma}
For a finite group $A$, $H^3(A; \mathbb{Z})=0$ if and only if $H_2(A; \mathbb{Z})=0$.
\end{lemma}
\begin{proof}
Since $A$ is finite, it is clear that $H^{>0}(A;\mathbb{Q})=0$. Then the long exact sequence of the cohomology associated to the short exact sequence of coefficients
\[
0\larrow \mathbb{Z}\larrow \mathbb{Q}\larrow \mathbb{Q}/\mathbb{Z}\larrow 0
\]
implies that $H^3(A;\mathbb{Z})\cong H^2(A;\mathbb{Q}/\mathbb{Z})$. For the latter cohomology group, the universal coefficient theorem implies that 
\[
H^2(A;\mathbb{Q}/\mathbb{Z})\cong {\rm Hom}(H_2(A;\mathbb{Z}), \mathbb{Q}/\mathbb{Z}) \oplus {\rm Ext} (H_1(A;\mathbb{Z}), \mathbb{Q}/\mathbb{Z}).
\]
Since $\mathbb{Q}/\mathbb{Z}$ is injective, ${\rm Ext} (H_1(A;\mathbb{Z}), \mathbb{Q}/\mathbb{Z})=0$ and then
\[
H^3(A;\mathbb{Z})\cong H^2(A;\mathbb{Q}/\mathbb{Z})\cong {\rm Hom}(H_2(A;\mathbb{Z}), \mathbb{Q}/\mathbb{Z}).
\] 
In particular, $H^3(A;\mathbb{Z})=0$ if and only if ${\rm Hom}(H_2(A;\mathbb{Z}), \mathbb{Q}/\mathbb{Z})=0$. However, it is clear that ${\rm Hom}(\mathbb{Z}, \mathbb{Q}/\mathbb{Z})$ and ${\rm Hom}(\mathbb{Z}/p^r, \mathbb{Q}/\mathbb{Z})$ with any prime $p$ are nontrivial. Therefore, ${\rm Hom}(H_2(A;\mathbb{Z}), \mathbb{Q}/\mathbb{Z})=0$ if and only $H_2(A;\mathbb{Z})=0$, and then the lemma follows. 
\end{proof}

Combining Lemmas \ref{H3lemma} and \ref{H3H2lemma}, we obtain the following lemma immediately.

\begin{lemma}\label{H2lemma}
If $H_2({\rm Syl}_2(G); \mathbb{Z})=0$, then the almost flat manifold $M$ is spin$^c$. ~$\qqed$
\end{lemma}

\begin{remark}\label{HG-rmk}
The subgroup inclusion ${\rm Syl}_2(G)\stackrel{i}{\hookrightarrow} G$ implies an odd-degree finite covering map $B{\rm Syl}_2(G)\stackrel{Bi}{\larrow} BG$. Then it is a homotopy equivalence after localization at the prime $2$. Therefore, the condition $H_2({\rm Syl}_2(G); \mathbb{Z})=0$ in Lemma \ref{H2lemma} is equivalent to that $H_2(G;\mathbb{Z})$ is of odd order. 
\end{remark}

In group theory, the second homology group $H_2(A;\mathbb{Z})$ of a finite group $A$ is known as the {\it Schur multiplier} of $A$, a classical notion introduced by Issai Schur in 1904. 

\begin{example}\label{Hopf-ex}
Suppose that $A$ has a presentation $A\cong F/R$ with $F$ a free group. Then the classical Hopf theorem \cite[Theorem II.5.3]{Bro82} implies that $H_2(A;\mathbb{Z})\cong R\cap [F, F]/[F, R]$. In particular, the Schur multiplier of $A$ is trivial if and only if $R\cap [F, F]= [F, R]$. 
\end{example}

Thanks to the classification of finite simple groups--a landmark achievement of 20th-century mathematics--the Schur multipliers of all finite simple groups have been completely determined; see, for instance, \cite{Wil09, Wik}. In what follows, we list all finite simple groups whose Schur multipliers have odd torsion.

\begin{example}\label{simple-ex}
Let $A$ be a finite simple group. The Schur multiplier $H_2(A;\mathbb{Z})$ has odd order precisely in the following cases:
\begin{itemize}
\item[(1).] Cyclic groups: $\mathbb{Z}/p$ with $p$ a prime number;
\item[(2).] Groups of Lie type (with $q=p^k$ for a prime number $p$ and positive $k$):
  \begin{itemize}
  \item Chevalley groups: 
     \begin{itemize} 
     \item[(Ch1).] linear groups $A_n(q)$ with $(n+1,q-1)$ odd and $(n, q)\neq (1, 4)$, $(2, 2)$, $(2, 4)$, $(3, 2)$; 
     \item[(Ch2).] orthogonal groups $B_n(q)$ with $n>1$, $q$ even and $(n, q)\neq (2, 2)$, $(3, 2)$; 
     \item[(Ch3).] symplectic groups $C_n(q)$ with $n>2$, $q$ even and $(n, q)\neq (3, 2)$; 
     \item[(Ch4).] orthogonal groups $D_n(q)$ with $n>3$, $(4,q^n-1)=1$ and $(n, q)\neq (4, 2)$; 
     \item[(Ch5).] $E_6(q)$; 
     \item[(Ch6).] $E_7(q)$ with $q$ even; 
     \item[(Ch7).] $E_8(q)$; 
     \item[(Ch8).] $F_4(q)$ with $q\neq 2$;
     \item[(Ch9).] $G_2(q)$ with $q\neq 4$;
     \end{itemize}
 \item Steinberg groups: 
   \begin{itemize} 
     \item[(S1).] unitary groups ${}^{2}A_n(q^2)$ with $n>1$, $(n+1,q+1)$ odd and $(n, q)\neq (3, 2)$, $(5,2)$; 
     \item[(S2).] orthogonal groups ${}^{2}D_n(q^2)$ with $n>3$, $(4,q^n+1)=1$; 
     \item[(S3).] ${}^{2}E_6(q^2)$ with $q\neq 2$;
     \item[(S4).] ${}^{3}D_4(q^3)$;
     \end{itemize}
 \item Suzuki groups: ${}^{2}B_2(2^{2n+1})$ with $n\geq 2$;
  \item Ree groups and Tits group: 
    \begin{itemize} 
     \item[(R1)] Ree groups ${}^{2}F_4(2^{2n+1})$ with $n\geq 1$; 
     \item[(R2)] Tits group ${}^{2}F_4(2)'$;
     \item[(R3)] Ree groups ${}^{2}G_2(3^{2n+1})$ with $n\geq 1$;
     \item[(R4)] the derived group ${}^{2}G_2(3)'$;
     \end{itemize}
     \end{itemize}
\item[(3).] Sporadic groups:
  \begin{itemize}
  \item Mathieu groups:   
     \begin{itemize} 
     \item[(M1)] $M_{11}$;
     \item[(M2)] $M_{23}$;
     \item[(M3)] $M_{24}$;
     \end{itemize}
  \item Janko groups:   
     \begin{itemize} 
     \item[(J1)] $J_{1}$;
     \item[(J2)] $J_{3}$;
     \item[(J3)] $J_{4}$;
     \end{itemize}
  \item Conway groups:   
     \begin{itemize} 
     \item[(C1)] $Co_{2}$;
     \item[(C2)] $Co_{3}$;
     \end{itemize}
  \item Fischer groups:
    \begin{itemize} 
     \item[(F1)] $Fi_{23}$;
     \item[(F2)] $Fi_{24}'$;
     \end{itemize}
  \item McLaughlin group: $McL$;
  \item Held group: $He$;
  \item O'Nan group: $O'N$;
  \item Harada-Norton group: $HN$;
  \item Lyons group: $Ly$;
  \item Thompson group: $Th$;
  \item Fischer-Griess Monster group, $M$.
   \end{itemize}
\end{itemize}
\end{example}

We are in a position to prove Theorem \ref{G-thm-intro}.
\begin{proof}[Proof of Theorem \ref{G-thm-intro}]
For the first statement, the Schur multiplier of $G$ is of odd order by assumption. By Remark \ref{HG-rmk}, this is equivalent to $H_2({\rm Syl}_2(G); \mathbb{Z})=0$. It follows that $M$ is spin$^c$ by Lemma \ref{H2lemma}, and hence $M$ bounds by Theorem \ref{spinc-thm-intro}.

For the second statement, the case when ${\rm Syl}_2(G)$ is cyclic or generalized quaternionic follows from Proposition \ref{DFprop}. For the remaining case when ${\rm Syl}_2(G)={\rm Syl}_2(A)$ for a finite simple group $A$ in Example \ref{simple-ex}, it is known that $H_2(A;\mathbb{Z})$ is of odd order. Then Remark \ref{HG-rmk} implies that $H_2({\rm Syl}_2(G); \mathbb{Z})=H_2({\rm Syl}_2(A); \mathbb{Z})=0$, or equivalently, the Schur multiplier of $G$ is of odd order. Consequently, the second statement in this case follows from the first statement. 
\end{proof}

%----------------------------------------------------------------------------------------------------------------------------------------------------------------------------------------------------------%
\section{The $4$-dimensional case}
\label{sec: 4}

Let $M$ be a $4$-dimensional closed orientable smooth manifold. A classical result of Wu \cite{Wu50} and Hirzebruch-Hopf \cite{HH58}, reproved by Teichner-Vogt \cite{TV} in modern language, shows that $M$ is spin$^c$. 

The following lemma is due to Szczepa\'{n}ski \cite[Lemma 1 with its proof and (1)]{Szc18}. 
Let $b_i(M)$ be the $i$-th Betti number of $M$.
\begin{lemma}\label{szc-lemma}
Let $M$ be an almost flat orientable $4$-manifold. If $M$ is non-spin, then $b_1(M)=1$ and $b_2(M)=0$. ~$\qqed$
\end{lemma}

\begin{proof}[Proof of Theorem \ref{4-thm-intro}]
By \cite{Wu50, HH58, TV}, the almost flat orientable $4$-manifold $M$ is spin$^c$. 

(1). Suppose that $M$ is non-spin. Then $b_2(M)=0$ by Lemma \ref{szc-lemma}. It follows that for any spin$^c$ structure $(M, c)$ on $M$, $c\in H^2(M;\mathbb{Z})$ is a torsion class. Then $\langle c^2, [M]\rangle=0$ and by Proposition \ref{cn=0-prop}, $M$ bounds a spin$^c$ manifold.

(2). Suppose that $M$ is spin. It is well known that the 4 dimensional spin bordism group $\Omega_4^{spin}=\mathbb{Z}$ and the generator of this group can be taken to be the $K3$ surface $X$, whose $\widehat{A}$-genus $\widehat{A}(X)=1$. Since $\widehat{A}(M)=-\frac{1}{24}\int_M p_1(M)=0$, $M$ must be 0 in $\Omega_4^{spin}$. 
\end{proof}

%----------------------------------------------------------------------------------------------------------------------------------------------------------------------------------------------------------%

%%% The bibliography %%%
\bibliographystyle{amsalpha}

\end{document}